\documentclass{amsart}

\usepackage[foot]{amsaddr}
\usepackage{graphicx}%
\usepackage{multirow}%
\usepackage{amsmath,amssymb,amsfonts}%
\usepackage{amsthm}%
\usepackage{mathrsfs}%
\usepackage[title]{appendix}%
\usepackage{xcolor}%
\usepackage{algorithm}%
\usepackage{listings}%
\usepackage{caption}
\usepackage{subcaption}
\usepackage{biblatex}

\newtheorem{thm}{Theorem}[section]
\newtheorem{prop}[thm]{Proposition}

\newtheorem{defn}[thm]{Definition}
\newtheorem{eg}[thm]{Example}

\newtheorem{rmk}[thm]{Remark}

\newcommand{\T}{\mathbb{T}}

\bibliography{refs.bib}

\title[A Geometric Conjecture and Some Properties of Blaschke Products]{Exploring a Geometric Conjecture, Some Properties of Blaschke Products, and the Geometry of Curves Formed by Them}

\author{Mehmet \c{C}elik$^{1, \dagger}$}
\address{$^1$ Department of Mathematics, Texas A\&M University-Commerce, 2200 Campbell Street, Commerce, 75428, TX, USA}

\author{Mathis Duguin$^{2, \dagger}$}
\address{$^2$ \'{E}cole Polytechnique F\'{e}d\'{e}rale de Lausanne, Rte Cantonale, Lausanne, 1015, Switzerland}

\author{Jia Guo$^{3, \dagger}$}
\address{$^3$ Department of Mathematics, University of Michigan, 530 Church St, Ann Arbor, 48109, MI, USA}

\author{Dianlun Luo$^{4, \dagger}$}
\address{$^4$ Department of Mathematics, Columbia University, 2990 Broadway, New York, 10027, New York, USA}

\author{Kamryn Spinelli$^{5, \dagger}$}
\address{$^5$ Department of Mathematics, Brandeis University, MS 050, 415 South Street, Waltham, 02453, MA, USA}

\author{Yunus E. Zeytuncu$^{6, \dagger}$}
\address{$^6$ Department of Mathematics and Statistics, University of Michigan-Dearborn, 4901 Evergreen Rd., Dearborn, 48128, MI, USA}

\author{Zhuoyu Zhu$^{3, \dagger}$}

\address{$^\dagger$ These authors contributed equally to this work.}

\begin{document}

\begin{abstract}
    In 2021, Dan Reznik made a YouTube video demonstrating that power circles of Poncelet triangles have an invariant total area. He made a simulation based on this observation and put forward a few conjectures. One of these conjectures suggests that the sum of the areas of three circles, each centered at the midpoint of a side of the Poncelet triangle and passing through the opposite vertex, remains constant. In this paper, we provide a proof of Reznik's conjecture and present a formula for calculating the total sum. Additionally, we demonstrate the algebraic structures formed by various sets of products and the geometric properties of polygons and ellipses created by these products.
\end{abstract}

\keywords{Finite Blaschke products, Poncelet ellipse, Blaschke ellipse, disk automorphism}

\subjclass{30J10, 53A04}

\maketitle

\section{Introduction}\label{sec:introduction}
Consider two ellipses, with one ellipse enclosed by the other, and a triangle inscribed in the larger ellipse which circumscribes the smaller ellipse. Then, by Poncelet’s theorem, each point on the outer ellipse is a vertex for one such triangle \cite[Theorem 5.6]{Gorkin}. In the case where the outer ellipse is a circle, the inner ellipses are known as Poncelet 3-ellipses and the circumscribing triangles are known as Poncelet triangles or Poncelet 3-periodics \cite[Chapter 5]{Gorkin}.

In one of his YouTube videos from 2021, Dan Reznik demonstrated that power circles of Poncelet triangles have invariant total area by creating a simulation of that observation and proposed it as a conjecture \cite{Reznik}. More precisely, consider a Poncelet $3$-ellipse with its major and minor axes on the $x$ and $y$ axes, and form three circles, each centered at the midpoint of each side of the Poncelet triangle and passing through the opposite vertex of the triangle; see Figure \ref{triangles}. The conjecture is that \textit{the sum of the areas of these three circles is constant}. This paper presents a proof of Reznik's conjecture (Theorem \ref{thm triangles}). The proof of Theorem \ref{thm triangles} involves the application of fundamental properties of Blaschke $3$-ellipses due to being precisely the Poncelet $3$-ellipses, \cite[Corollary 5.8]{Gorkin}. A formula is constructed to represent the sum of the areas of the three circles, and it is shown that this sum is constant and depends only on the Poncelet ellipse's foci.
\begin{figure}[htp]
        \centering
        \includegraphics[width=8cm]{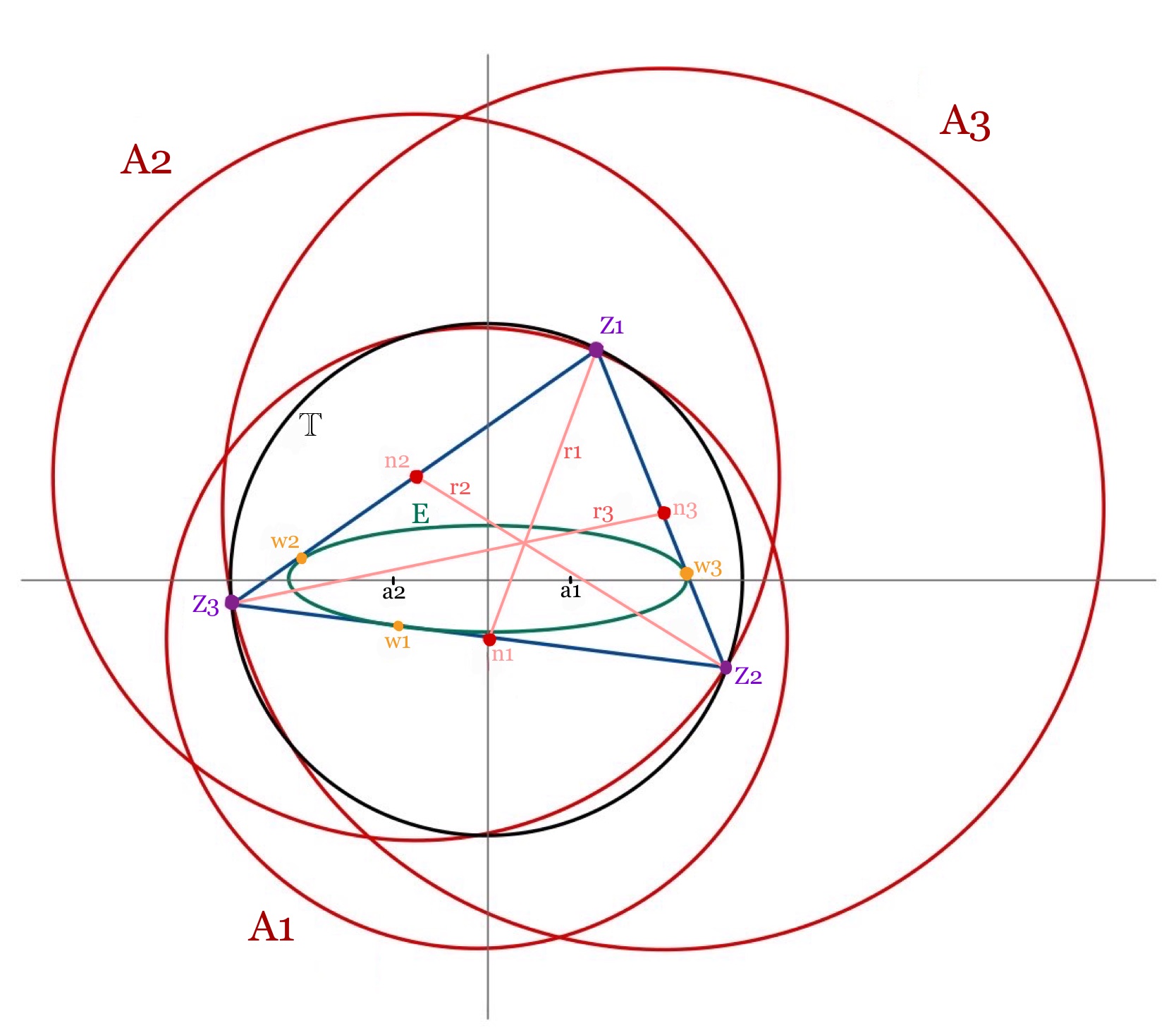}
        \caption{The three circles are separately centered at the midpoint of each side of the Poncelet triangle and pass through the opposite vertex of the triangle.}
        \label{triangles}
\end{figure}

To introduce the properties of Blaschke 3-ellipses, used in the proof of Theorem \ref{thm triangles}, and to share other observations associated with the Blaschke product and associated ellipses later in the paper, let's first introduce some definitions and notions. 
\begin{defn}
A finite Blaschke product of degree $n\in\mathbb{N}$ is a complex function of the form 
\begin{align}\label{Blaschke}
B(z)=\mu\prod_{j=1}^n\frac{z-a_j}{1-\overline{a_j}z}=\mu\prod_{j=1}^n\sigma_{a_j}(z)    
\end{align}
for $a_1,...,a_n\in\mathbb{D}$ where $\mathbb{D}$ is the open unit disk and $\mu\in\mathbb{T}$ where $\mathbb{T}$ is the unit circle.
\end{defn}
A M\"{o}bius transformation, also known as a linear fractional transformation, is a map of the form $f(z)=\frac{az+b}{cz+d}$ where $a$, $b$, $c$, $d\in\mathbb{C}$ and $ad-bc\not=0$. 
These transformations form a group called the M\"{o}bius group, also known as the projective general linear group $PGL(2,\mathbb{C})$. 
Each factor of a Blaschke product is a M\"{o}bius transformation which has a simple zero in $\mathbb{D}$, is holomorphic on a neighborhood of $\overline{\mathbb{D}}$, and maps the unit circle to itself. Accordingly, a Blaschke product maps the unit circle to itself and the unit disk to itself. The number of preimages of any point on the unit circle equals the degree of the Blaschke product. For more on the theory of Blaschke products, see, for example, \cite{Garcia}.

Blaschke products are prototypes for complex analytic functions mapping the unit disk $\mathbb{T}$ to itself. Given such a function $f$, it is possible to uniformly approximate $f$ on compact subsets of $\mathbb{D}$ using finite Blaschke products: see, for example, Carath\'{e}odory's theorem in \cite{Garnett}. In our work, we are concerned with geometric properties associated with certain curves in $\mathbb{D}$ by the Blaschke product via its preimages of points on $\mathbb{T}$.
 
Suppose a Blaschke product has degree $3$ with one of the zeros at the origin and the other two at $a$ and $b$. Then according to \cite[Theorem 2.9]{Gorkin}, for any point $\lambda$ on the unit circle, the preimages of $\lambda$ under the Blaschke product are three different points on $\mathbb{T}$ forming a triangle where each side of the triangle is tangent to the ellipse given by 
\begin{align}\label{formula-B3ellipse}
|w-a|+|w-b| = |1-\overline{a}b|.
\end{align}
The same result, \cite[Theorem 2.9]{Gorkin}, also offers a converse: Every point on the ellipse is a point of tangency of a line segment that intersects the unit circle at two points mapped to the same value under the Blaschke product. This ellipse is called a \textit{\textbf{Blaschke $3$-ellipse}}. 
    
There is a similar result for degree-$4$ Blaschke products, with some restrictions to make sure the ellipse exists; see \cite[Theorem 3.6]{GorkinWagner2017}.
Consider a degree-$3$ Blaschke product denoted by $B$, with zeros at $a$, $b$, and $c$. If $B_1(z) = zB(z)$ can be expressed as $C(D(z))$, where $C$ and $D$ are degree-$2$ normalized Blaschke products (meaning that $C(0) = 0$ and $D(0) = 0$), and $a$ corresponds to the zero of $D(z)/z$, then the Poncelet curve associated with $B_1$ can be expressed in the form 
    \begin{align}\label{formula-B4ellipse}
    |w-b|+|w-c| = |1-\bar{b}c|\sqrt{\frac{2 - |b|^{2} - |c|^{2}}{1 - |b|^{2}|c|^{2}}}.
    \end{align}   
If a Blaschke product can be written as a composition with Blaschke products of lower degrees but degrees greater than $1$, we say that the Blaschke product is decomposable. Thus, the Blaschke curve associated with a degree-$4$ Blaschke product is an ellipse if and only if the Blaschke product is decomposable, \cite[Corollary 13.9]{Gorkin}. 
 
In this paper, our primary focus is on proving Dan Reznik's conjecture. We utilize the fundamental properties of Blaschke 3-ellipses, which are precisely the Poncelet 3-ellipses, and examine the geometric attributes of Blaschke products via this correspondence. Furthermore, we delve into the algebraic and geometric features of specific types of Blaschke products. The paper is structured as follows. Section \ref{sec:reznik-conjecture-proof} states and proves Reznik's conjecture on areas of power circles of Poncelet triangles, mentioned at the beginning of the paper. Section \ref{sec:other-observations} explores some observations on curvatures of Blaschke ellipses and remarks on the geometry of reducible (and semi-reducible) Blaschke products. Lastly, Section \ref{section with concluding remarks} presents some comments, conjectures, and future directions to explore the project further. 

\section{Proof of Reznik's conjecture} \label{sec:reznik-conjecture-proof}
As mentioned in Section \ref{sec:introduction}, the Blaschke $3$-ellipses are precisely the Poncelet $3$-ellipses. We will leverage this fact to prove Dan Reznik's conjecture regarding the combined area of three circles associated with a Poncelet 3-triangle. Additionally, we will show that the sum of the circle areas is solely determined by the foci of the Poncelet ellipse.

\begin{thm}[Reznik's conjecture, \cite{Reznik}] \label{thm triangles} Let $E$ be a Poncelet 3-ellipse whose center coincides with that of the surrounding circle $C$. For each triangle circumscribing $E$ and inscribed in $C$, consider three circles, separately centered at the midpoint of each side of the triangle and passing through the opposite vertex of the triangle. Then the sum of the areas of these three circles is the same for any such triangle, i.e. it is an invariant of the pair $(E, C)$.
\end{thm}
Figure \ref{triangles} shows a graph that explains what we are trying to prove in the case where $C$ is the unit circle $\mathbb{T}$.
\begin{itemize}
    \item $\mathbb{T}$ is the unit circle.
    \item $E$ is a Poncelet $3$-ellipse centered at the origin with foci $a_1$ and $a_2$ on the real axis.
    \item The triangle is simultaneously inscribed in the unit circle $\mathbb{T}$ and circumscribing the ellipse $E$.
    \item $z_1$, $z_2$, and $z_3$ are the vertices of the triangle on the unit circle.
    \item $w_1$, $w_2$, and $w_3$ are the points of tangency of the ellipse $E$ and the sides of the triangle.
    \item $n_1$, $n_2$, and $n_3$ are the midpoints of each side of the triangle.
    \item $A_j$ is a circle centered at $n_j$ with radius equal to the length $r_j=|z_j-n_j|$ for $j=1,2,3$.
\end{itemize}
In the proof of Theorem \ref{thm triangles}, we primarily used the elimination method, which involves a system of equations connecting the degree-3 Blaschke product and the Poncelet 3-ellipse. We examine the implications of the Blaschke product in \eqref{eq201} and \eqref{eq202}, which lead to the derivation of two main equations, \eqref{eq101} and \eqref{eq102}. By combining and simplifying these equations, we obtain an equation that depends only on the ellipse's foci. This equation demonstrates that the sum of the area remains constant for a fixed Poncelet $3$-ellipse.

\begin{proof}[Proof of Theorem \ref{thm triangles}]
    Let $E$ be a Poncelet 3-ellipse whose center coincides with that of the surrounding circle $C$. By applying a translation and a dilation, which will simply change all areas by a constant factor, we may assume without loss of generality that $C$ is the unit circle $\mathbb{T}$ and that $E$ is centered at the origin with major and minor axes lying along the $x$ and $y$ axes of the plane respectively. Let $B$ be the corresponding degree-$3$ Blaschke product with zeros located at $0$, $a_1$, and $a_2$. $a_1$ and $a_2$ will be the foci of the ellipse $E$, and, under our assumption, they will be real. Then for each $\lambda \in \mathbb{T}$, there are exactly $3$ distinct solutions to $B(z)=\lambda$ and these solutions lie on the unit circle $\mathbb{T}$, \cite[Lemma 3.4]{Gorkin}. Moreover, if $\lambda \in \mathbb{T}$ and the three points $B$ sends to $\lambda$ are $z_{1}, z_{2}$, and $z_{3}$, then, as in \cite[pg.39]{Gorkin},
    \begin{align}\label{eq201}
    B(z)-\lambda=\frac{\left(z-z_{1}\right)\left(z-z_{2}\right)\left(z-z_{3}\right)}{\left(1-\overline{a_{1}} z\right)\left(1-\overline{a_{2}} z\right)},
    \end{align}
    and evaluating equation \eqref{eq201} at $a_{j}$ for $j=1,2$  yields
    \begin{align}\label{eq202}
    1=|\lambda|=\left|\frac{\left(a_{j}-z_{1}\right)\left(a_{j}-z_{2}\right)\left(a_{j}-z_{3}\right)}{\left(1-\left|a_{j}\right|^{2}\right)\left(1 - a_j \overline{a_{j \pm 1}}\right)}\right|.
    \end{align}
    The sum of the areas $A_1$, $A_2$, and $A_3$ is
    \[A_1+A_2+A_3=\pi r_1^{2}+\pi r_{2}^{2}+\pi r_{3}^{2}
    =\pi\left(r_{1}^{2}+r_{2}^{2}+r_{3}^{2}\right).\]
    Let's rewrite the sum of the areas in terms of the vertices $z_1$, $z_2$, and $z_3$ of the triangle.
    \begin{align*}
    r_1=\left|z_{1}-\frac{1}{2}\left(z_{2}+z_{3}\right)\right|,\ 
    r_{2}=\left|z_{2}-\frac{1}{2}\left(z_{1}+z_{3}\right)\right|, \
    r_{3}=\left|z_{3}-\frac{1}{2}\left(z_{1}+z_{2}\right)\right|.    
    \end{align*}
    Note that the vertices $z_1$, $z_2$, and $z_3$ are on the unit circle. So $|z_1|=|z_2|=|z_3|=1$ and
    \begin{align*}
    r_3^2 &=\frac{3}{2}-\frac{1}{2}\left(\bar{z}_{1} z_{3}+z_{1} \bar{z}_{3}+\bar{z}_{2} z_{3}+z_{2} \bar{z}_{3}\right)+\frac{1}{4}\left(\bar{z}_{1} z_{2}+z_{1} \bar{z_{2}}\right),
    \end{align*}
    \begin{align*}
    r_{2}^{2}&=
    \frac{3}{2}-\frac{1}{2}\left(\bar{z}_{1} z_{2}+z_{1} \bar{z}_{2}+\bar{z}_{2} z_{3}+z_{2} \bar{z}_{3}\right)+\frac{1}{4}\left(\bar{z}_{1} z_{3}+z_{1} \bar{z}_{3}\right),
    \end{align*}
    \begin{align*}
    r_{1}^{2}&=
    \frac{3}{2}-\frac{1}{2}\left(\bar{z}_{1} z_{2}+z_{1} \bar{z}_{2}+z_{1} \bar{z}_{3}+\bar{z}_{1} z_{3}\right)+\frac{1}{4}\left(\bar{z}_{2} z_{3}+z_{2} \bar{z}_{3}\right).
    \end{align*}
    Thus,
    \begin{align}\label{sum of areas}
    r_{1}^{2}+r_{2}^{2}+r_{3}^{2}=\frac{9}{2}-\frac{3}{4}\left(\bar{z_{1}} z_{2}+z_{1} \bar{z}_{2}+\bar{z}_{2} z_{3}+z_{2} \overline{z}_{3}+\bar{z}_{1} z_{3}+z_{1} \bar{z}_{3}\right).   
    \end{align}
    To conclude that the sum of the areas $A_1$, $A_2$, and $A_3$ is invariant, it suffices to show that the expression $\left(\bar{z_{1}} z_{2}+z_{1} \bar{z}_{2}+\bar{z}_{2} z_{3}+z_{2} \overline{z}_{3}+\bar{z}_{1} z_{3}+z_{1} \bar{z}_{3}\right)$ in \eqref{sum of areas} is constant as the vertices vary over all circumscribing triangles.

    Let $a_1$ and $a_2$ be the two foci of the ellipse $E$, and assume without loss of generality that $a_1>a_2$. Because $E$ is centered at the origin and $a_1$ and $a_2$ are on the real axes, we have $a_1=\overline{a_1}$, $a_2=\overline{a_2}$, and $a_1=-a_2$.
    
    For $j=1, 2$ in \eqref{eq202} with $a_2=-a_1$ we have
    \begin{align}
      \left|\left(z_{1}-a_{1}\right)\left(z_{2}-a_{1}\right)\left(z_{3}-a_{1}\right)\right|^2&=\left(1-a_{1}^{2}\right)^{2}\left(1+a_{1}^{2}\right)^{2}\label{eq101}\\
     \left|\left(z_{1}+a_{1}\right)\left(z_{2}+a_{1}\right)\left(z_{3}+a_{1}\right)\right|^2&=\left(1-a_{1}^{2}\right)^{2}\left(1+a_{1}^{2}\right)^{2}\label{eq102}   
    \end{align}
    Expanding the left-hand-sides of equations \eqref{eq101} and \eqref{eq102} we obtain
    \begin{align}\label{eq103}
    \nonumber&\left(1+a_{1}^{2}\right)^{3}-a_{1}\left(1+a_{1}^{2}\right)^{2}\left(z_{1}+\bar{z}_{1}\right)-a_{1}\left(1+a_{1}^{2}\right)^{2}\left(z_{2}+\bar{z}_{2}\right)\\
    \nonumber&-a_{1}\left(1+a_{1}^{2}\right)\left(z_{3}+\bar{z}_{3}\right)
    +a_{1}^{2}\left(1+a_{1}^{2}\right)\left(z_{1}+\bar{z}_{1}\right)\left(z_{2}+\bar{z}_{2}\right)\\
    \nonumber&+a_{1}^{2}\left(1+a_{1}^{2}\right)\left(z_{1}+\bar{z}_{1}\right)\left(z_{3}+\bar{z}_{3}\right)+a_{1}^{2}\left(1+a_{1}^{2}\right)\left(z_{2}+\bar{z}_{2}\right)\left(z_{3}+\bar{z}_{3}\right)\\
    &-a_{1}^{3}\left(z_{1}+\bar{z}_{1}\right)\left(z_{2}
    +\bar{z}_{2}\right)\left(z_{3}+\bar{z}_{3}\right)    
    \end{align}
    and
    \begin{align}\label{eq104}
    \nonumber&\left(1+a_{1}^{2}\right)^{3}+a_{1}\left(1+a_{1}^{2}\right)^{2}\left(z_{1}+\bar{z}_{1}\right)+a_{1}\left(1+a_{1}^{2}\right)^{2}\left(z_{2}+\bar{z}_{2}\right)\\
    \nonumber&+a_{1}\left(1+a_{1}^{2}\right)\left(z_{3}+\bar{z}_{3}\right)+a_{1}^{2}\left(1+a_{1}^{2}\right)\left(z_{1}+\bar{z}_{1}\right)\left(z_{2}+\bar{z}_{2}\right)\\
    \nonumber&+a_{1}^{2}\left(1+a_{1}^{2}\right)\left(z_{1}+\bar{z}_{1}\right)\left(z_{3}+\bar{z}_{3}\right)+a_{1}^{2}\left(1+a_{1}^{2}\right)\left(z_{2}+\bar{z}_{2}\right)\left(z_{3}+\bar{z}_{3}\right)\\
    &+a_{1}^{3}\left(z_{1}+\bar{z}_{1}\right)\left(z_{2}+\bar{z}_{2}\right)\left(z_{3}+\bar{z}_{3}\right).
    \end{align}
    
    The right-hand side expressions of equations \eqref{eq101} and \eqref{eq102} are identical, so the left-hand sides of equations \eqref{eq103} and \eqref{eq104} are equal. More specifically, 
    \begin{align}
    \nonumber a_{1}\left(1+a_{1}^{2}\right)^{2}\left(z_{1}+\bar{z}_{1}\right)+a_{1}&\left(1+a_{1}^{2}\right)^{2}\left(z_{2}+\bar{z}_{2}\right)+a_{1}\left(1+a_{1}^{2}\right)\left(z_{3}+\bar{z}_{3}\right)\\
    &+a_{1}^{3}\left(z_{1}+\bar{z}_{1}\right)\left(z_{2}+\bar{z}_{2}\right)\left(z_{3}+\bar{z}_{3}\right)=0.    
    \end{align}
    Therefore, the remaining terms in \eqref{eq103} and \eqref{eq104} sum to $\left(1-a_{1}^{2}\right)^{2}\left(1+a_{1}^{2}\right)^{2}$. That is,
    \begin{align}\label{eq105}
    \nonumber \left(1+a_{1}^{2}\right)^{3}&+a_{1}^{2}\left(1+a_{1}^{2}\right)\left(z_{1}+\bar{z}_{1}\right)\left(z_{2}+\bar{z}_{2}\right)+a_{1}^{2}\left(1+a_{1}^{2}\right)\left(z_{1}+\bar{z}_{1}\right)\left(z_{3}+\bar{z}_{3}\right)\\
    &+a_{1}^{2}\left(1+a_{1}^{2}\right)\left(z_{2}+\bar{z}_{2}\right)\left(z_{3}+\bar{z}_{3}\right)=\left(1-a_{1}^{2}\right)^{2}\left(1+a_{1}^{2}\right)^{2}.    
    \end{align}
    One can then simplify \eqref{eq105} immediately by cancelling one of the $(1+a_1^2)$, subtracting the remaining $(1+a_1^2)^2$ from both sides, and cancelling an $a_1^2$. From this, it follows that
    \begin{align}\label{eq105a}
        (z_1+\overline{z}_1)(z_2+\overline{z}_2)+(z_1+\overline{z}_1)(z_3+\overline{z}_3)+(z_2+\overline{z}_2)(z_3+\overline{z}_3)=-3+a_1^4-2a_1^2.
    \end{align}
    Recall that we need to compute $\bar{z_{1}} z_{2}+z_{1} \bar{z}_{2}+\bar{z}_{2} z_{3}+z_{2} \overline{z}_{3}+\bar{z}_{1} z_{3}+z_{1} \bar{z}_{3}=:A$ in \eqref{sum of areas}.
    However, notice that \eqref{eq105a} can be written as 
    \begin{align}
    z_1z_2+\overline{z}_1\overline{z}_2+z_1z_3+\overline{z}_1\overline{z}_3+z_2z_3+\overline{z}_2\overline{z}_3+A=-3+a_1^4-2a_1^2.
    \end{align}
    Using equation \eqref{eq201} and evaluating at $0$, we obtain $z_1z_2z_3=\lambda$. Therefore, $z_1=\lambda \overline{z}_2\overline{z}_3$ and similarly for $z_2$ and $z_3$. This leads to 
    \begin{align}\label{eq105b}    z_1z_2+\overline{z}_1\overline{z}_2+z_1z_3+\overline{z}_1\overline{z}_3+z_2z_3+\overline{z}_2\overline{z}_3=\lambda(\overline{z}_1+\overline{z}_2+\overline{z}_3)+\overline{\lambda}(z_1+z_2+z_3).
    \end{align}
    Let's find what is $(z_1 + z_2 + z_3)$ in equation \eqref{eq105b}. Based on equation \eqref{eq201}, $(z_1 + z_2 + z_3)$ represents the negative of the coefficient of $z^2$ when expanding $(z-z_1)(z-z_2)(z-z_3)$. However $a_2=-a_1$, so 
    \begin{align*}
        (B(z)-\lambda)(1-\overline{a}_1^2z^2)=(z-z_1)(z-z_2)(z-z_3).
    \end{align*}
    Since $a_1$ is real and $z^2$ doesn't appear in the numerator of 
    \begin{align*}
        B(z)=\frac{z(z-a_1)(z+a_1)}{(1-\overline{a}_1^2z^2)}=\frac{z(z^2-a_1^2)}{(1-\overline{a}_1^2z^2)}
    \end{align*}
    the coefficient is $\lambda a_1^2$; that is, $z_1z_2+\overline{z}_1\overline{z}_2+z_1z_3+\overline{z}_1\overline{z}_3+z_2z_3+\overline{z}_2\overline{z}_3=-2a_1^2$. So $A=-3+a_1^4$. Substituting this into \eqref{sum of areas} yields
    \begin{align*}
        r_1^2+r_2^2+r_3^2=\frac{9}{2}-\frac{3}{4}(a_1^4-3)
    \end{align*}
    which is a constant depending only on the foci of the ellipse, thus proving the theorem.
\end{proof}

\begin{rmk}\label{examples of ellipses not centered at the origin}
Theorem \ref{thm triangles} for an arbitrary Poncelet $3$-ellipse is false. 
For example, consider the ellipse formulated as 
\[\left|w-\left(\frac{\sqrt{2}}{3}-\frac{1}{3}i\right)\right|+\left|w-\left(-\frac{\sqrt{2}}{3}-\frac{1}{3}i\right) \right| = \frac{2}{\sqrt{3}}.\] 
When the Poncelet triangles change, the sum of the areas of the three circles also changes. This is demonstrated in the applet ``Power circles of Poncelet 3-periodics with changing total area"\footnote{Created with GeoGebra, by Dianlun Luo, {\tt https://www.geogebra.org/m/nvvramgf}}.\\
On the other hand, one can observe some examples of Poncelet $3$-ellipses that when one of the ellipse foci is at the origin, the sum of the areas $A_1$, $A_2$, and $A_3$ is still invariant, such as
$\left|w\right|+\left|w-0.5 \right| = 1$ and $\left|w\right|+\left|w-0.1 \right| = 1$.
This is demonstrated in the applet ``Poncelet 3-Ellipse (one foci at the origin)"\footnote{Created with GeoGebra, by Dianlun Luo,  {\tt https://www.geogebra.org/m/t7tcwtyf}}.

We propose the following question: can the result of Theorem \ref{thm triangles} be generalized to Poncelet $n$-ellipses (in other words, ellipses inscribed in $n$-gons which are themselves inscribed in $\T$), and under what conditions?
\end{rmk}

\section{Other observations on geometry and decomposition of Blaschke products}\label{sec:other-observations}
\subsection{Analysis of curvature properties of Blaschke 3- and 4-ellipses}

In this section, we first focus on the geometric features of degree-$3$ and degree-$4$ Blaschke products. In particular, we utilize the geometric connection between Poncelet ellipses and Blaschke products mentioned in Section \ref{sec:introduction} to explore geometric properties of Blaschke ellipses.

The curvature function $\kappa$ of Blaschke $3$-ellipses and Blaschke $4$-ellipses with the same foci is formulated as
\begin{align}\label{curveture34}
\kappa(t) = \frac{Mm}{4((\frac{M}{2})^{2}\sin(t + \theta)^{2} + (\frac{m}{2})^{2}\cos(t + \theta)^{2})^{\frac{3}{2}}},
\end{align} 
where $M$ is the length of the major axis, $m$ is the length of the minor axis, and $\theta$ is the angle between the major axis and the $x$-axis. 
The equation for an ellipse that is centered at the origin and parallel to the $x$-axis, with $M$ as the length of the major axis and $m$ as the length of the minor axis, is given by 
$4x^{2}/M^{2} + 4y^{2}/m^{2} = 1.$
By parameterizing this ellipse as $a(t)=(x(t),y(t))=\left(M(e^{it}+e^{-it})/4,
m(e^{it}-e^{-it})/(4i)\right)$ where $0 \le t \le 2\pi$, then rotating the parametrized form by an angle $\theta$ and using well-known formulas for the curvature of a smooth curve, we arrive at the general expression in \eqref{curveture34}.
The general formulation for the curvature $\kappa$ in \eqref{curveture34} (as well as the eccentricity $e = \sqrt{1 - |a_{2} - a_{1}|^{2}/M^{2}}$, where $a_1$, $a_2$ are the foci of the ellipse) of a degree-$3$ and a degree-$4$ Blaschke ellipse is the same except for the major and minor axes' expressions for length.

The following example compares the curvature functions of two ellipses with the same foci, which are the Blaschke ellipses of a degree-$3$ Blaschke product and degree-$4$ Blaschke product respectively.
\begin{eg}\label{curvature-compare}
\begin{figure}[htp]
        \centering
        \includegraphics[width=10cm]{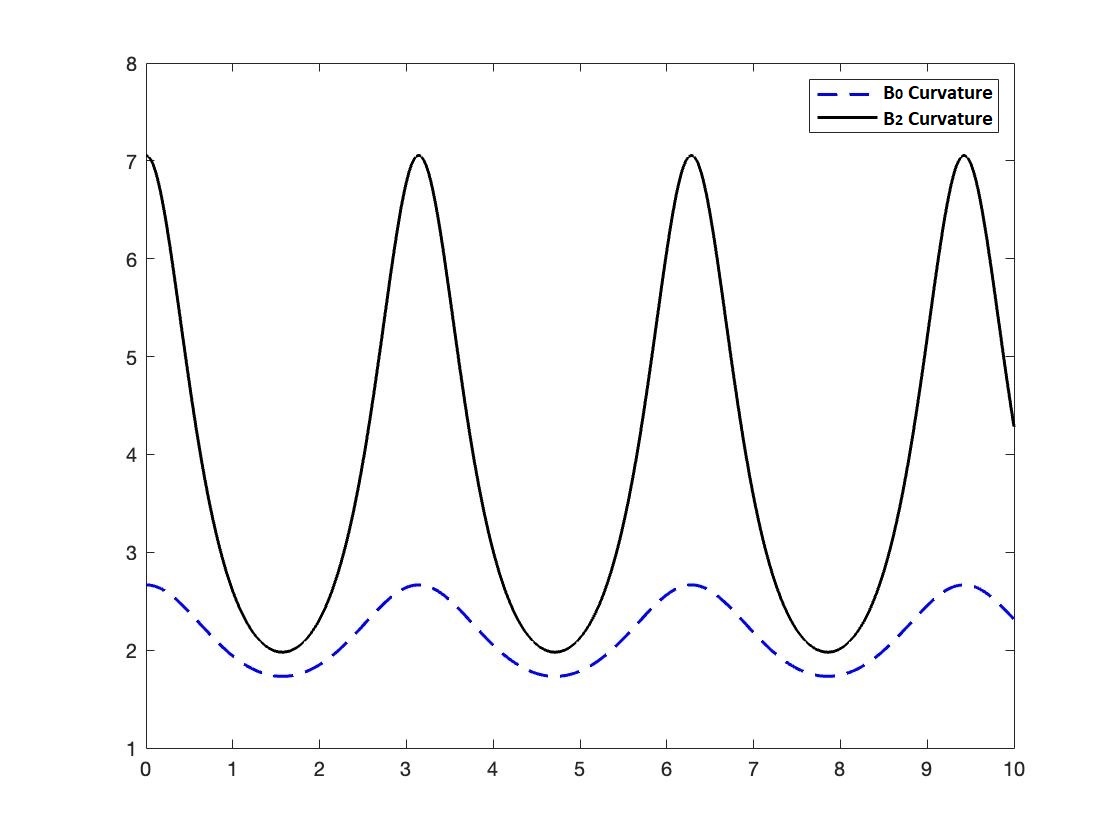}
        \caption{The curvature functions of the degree-$3$ Blaschke ellipse and degree-$4$ Blaschke ellipse (blue dashed line and solid black line, respectively).}
    \label{Curvetures34B}
\end{figure}
Consider a Blaschke $3$-ellipse generated by $B_{0}(z)=z^{2}\left(\frac{z-0.5}{1-0.5z}\right)$. Its foci are $a_1=0$, $a_2=0.5$, its major and minor axis lengths are $M=|1-\bar{a_{1}}a_{2}|=|1-0\cdot0.5| = 1$, $m =\sqrt{M^2 - |a_2 - a_1|^{2}}=\sqrt{1 - 0.25} = \frac{\sqrt{3}}{2}$, and the angle between the major axis and the $x$-axis is $\theta= \arccos\left(\frac{0.5 - 0}{|0.5 - 0|}\right) = 0$.
Therefore, by \eqref{curveture34} the curvature of this Blaschke $3$-ellipse is 
$\kappa(t) = \frac{\sqrt{3}}{(\sin(t)^{2} + \frac{3}{4}\cos(t)^{2})^{\frac{3}{2}}}.$

On the other hand, let $B_{1}(z) = z\left(\frac{z-0.5}{1-0.5z}\right)^{2}$, a degree-3 Blaschke product with zeros $a = 0.5$, $b = 0.5$, and $c = 0$. Let also $B_{2}(z)=z B_1(z) = z^{2}\left(\frac{z-0.5}{1-0.5z}\right)^{2}$, a degree-$4$ Blaschke product. Observe that $B_{2}(z)$ can be written as the composition of two normalized degree-$2$ Blaschke products $B_2(z)=C(D(z))$, where $C(z)= z^{2}$ and $D(z) = z(\frac{z - 0.5}{1 - 0.5z})$. Then, according to Theorem 3.6 in \cite{GorkinWagner2017}, the Blaschke curve associated with $B_2$ forms an ellipse with foci $b = 0$ and $c = 0.5$. The major and minor axis lengths of this ellipse are $M=|1-0.5|\sqrt{\frac{2 - |0|^{2} - |0.5|^{2}}{1 - |0|^{2}|0.5|^{2}}} = \frac{\sqrt{7}}{4}$ and $m =\sqrt{(\frac{\sqrt{7}}{4})^{2} - |0.5 - 0|^{2}} = \frac{\sqrt{3}}{4}$, and the angle between the major axis and the $x$-axis is $\theta = \arccos\left(\frac{0.5 - 0}{|0.5 - 0|}\right) = 0$. By \eqref{curveture34}, the curvature of this Blaschke $4$-ellipse is 
$\kappa(t)=\frac{8\sqrt{21}}{(7\sin(t)^{2} + 3\cos(t)^{2})^{\frac{3}{2}}}$.

The graphs of the curvature functions of the Blaschke ellipses of $B_{0}$ and $B_{2}$ are displayed in Figure \ref{Curvetures34B}. Notice that although the extreme values of these functions coincide because both ellipses have the same foci, the curvature function of the Blaschke $3$-ellipse is different than that of the Blaschke $4$-ellipse in values, rate of change, and shape.
\end{eg}

\begin{rmk}\label{upper-lower bound}
The upper and lower bounds of the curvature of a Blaschke-$3$ ellipse depend on the major and minor axis lengths. Since the size of the axes depends on the ellipse's foci, the upper and lower bounds of the curvature function thus depend on the zeros of the Blaschke product. One can rewrite the denominator of $\kappa(t)$ in \eqref{curveture34} as
\[4\left(\left(\frac{M}{2}\right)^{2} - \left(\left(\frac{M}{2}\right)^{2} - \left(\frac{m}{2}\right)^{2}\right)\cos^2(t + \theta)\right)^{3/2}.\]
This can be bounded above by $\frac{M^{3}}{2}$ since $m \leq M$, to obtain
\[\kappa(t)\geq \frac{Mm}{\frac{M^{3}}{2}} = \frac{2m}{M^{2}}.\]
Similarly, one can rewrite the denominator of $\kappa(t)$ as
\[4\left(\left(\frac{m}{2}\right)^{2} + \left(\left(\frac{M}{2}\right)^{2} - \left(\frac{m}{2}\right)^{2}\right)\sin^2(t + \theta)\right)^{3/2}\] 
which can be bounded below by $\frac{m^{3}}{2}$ since $m \leq M$. This yields
\[\kappa(t)\leq \frac{Mm}{\frac{m^{3}}{2}} = \frac{2M}{m^{2}}.\]
\end{rmk}

The following example shows that the upper and lower bounds of the Blaschke $3$-ellipse in Remark \ref{upper-lower bound} curvature are attainable.

\begin{eg}
Consider the ellipse generated by the degree-$3$ Blaschke product from Example \ref{curvature-compare},  $B_{0}(z)=z^{2}\left(\frac{z-0.5}{1-0.5z}\right)$. The curvature function of this ellipse is 
$\kappa(t)= \frac{\sqrt{3}}{(\sin(t)^{2} + \frac{3}{4}\cos(t)^{2})^{\frac{3}{2}}}$.
Using techniques from calculus, one find that the absolute minimum and maximum values are $\sqrt{3}$ and $8/3$, consecutively. These are exactly the lower and upper bounds asserted in Remark \ref{upper-lower bound}.
\end{eg}

\subsection{Decomposition of reducible Blaschke products}
We also explored the decomposition of reducible Blaschke products, which are conjugations of power functions by disk automorphisms (which are themselves the simplest examples of Blaschke products).

\begin{defn}\label{def-reducible}
A degree-$n$ Blaschke product $B$ is said to be \textbf{reducible} if a normalized automorphism $\sigma_a(z)=\frac{z-a}{1-\overline{a}{z}}$ exists, such that 
\[B=\sigma_a\circ h \circ \sigma_a^{-1}\]
where $h(z)=z^n$.
In this case, $\sigma_a$ is called the conjugate factor of $B$.
\end{defn}

Pre-composing or post-composing a Blaschke product with a disk automorphism results in a Blaschke product with the same degree.
Consider a degree-$n$ Blaschke product $B(z)=\prod_{i=1}^n\sigma_{a_j}(z)$ and a disk automorphism  $\sigma_a(z)=\frac{z-a}{1-\overline{a}z}$. 
The compositions $B \circ \sigma_a$ and $\sigma_a\circ B$ are Blaschke products of the same degree as $B$:
both are holomorphic on a neighborhood of $\overline{\mathbb{D}}$ and map unit circle to unit circle. Moreover, $B\circ\sigma_a$ has zeros $\{\sigma_a^{-1}(a_j)\}_{j=1}^{n}$, and $\sigma_a\circ B$ has all $n$ preimages of $a$ under $B$ as zeros.

We will denote the set of $n$-th roots of a complex number $a$ by $\Sigma_{a}^n:=\{\zeta_n^jf(a)\}_{j=1}^n$
where $f$ is the principal branch of the $n$-th root on $\mathbb{C}\setminus\mathbb{R}_{-}$ and $\zeta_n$ is a primitive $n$-th root of unity. We also say that a set of points is \textit{concyclic} if they are on the same circle.
\begin{prop}\label{zeros-reducible-B}
    Let $B$ be a degree-$n$ reducible Blaschke product with conjugate factor $\sigma_a$. Then the set of zeroes of $B$ is $\sigma_a(\Sigma_a^n)$. In particular, the zeroes of $B$ are concyclic and distinct as long as $B\neq h$.
\end{prop}
\begin{proof}
Let $z$ be a zero of $B=\sigma_a\circ h\circ\sigma_a^{-1}$. Then $h\circ\sigma_a^{-1}(z)=a$, so $\sigma_a^{-1}(z)\in h^{-1}(a)$. Since we have $h^{-1}(a)=\Sigma_a^n$, we get that $z\in\sigma_a(\Sigma_a^n)$.

Conversely, any element $w\in\sigma_a(\Sigma_a^n)$ can be written as $\sigma_a(\zeta_n^jf(a))$, with $f$ as in the paragraph before Proposition \ref{zeros-reducible-B}, for some $j$. Then $B(w)=(\sigma_a\circ h\circ\sigma_a^{-1})(\sigma_a(\zeta_n^jf(a)))=\sigma_a(a)=0$.
Since $\Sigma_a^n$ is a subset of a circle, we must have that $\sigma_a(\Sigma_a^n)$ is a subset of a circle on the Riemann sphere. However, $\sigma_a(\Sigma_a^n)$ cannot be a line since $\sigma_a(\mathbb{D})=\mathbb{D}$. Thus $\sigma_a(\Sigma_a^n)$ lies on a circle in $\mathbb{C}$ as required. Since $\Sigma_a^n$ has $n$ distinct elements as long as $a\neq 0$, we conclude the last point by injectivity of $\sigma_a$.
\end{proof}

We remark that in this setup, the preimage of $0$ is not unusual: the same argument could be applied to show that the preimage of any point under a reducible Blaschke product must be a concyclic set. We cannot hope for this to be a sufficient condition to characterize reducible Blaschke products, however, for this is true in the case of every degree-$3$ product.

After looking at the zeroes of a reducible Blaschke product, we may turn to analyzing its critical points.
\begin{prop}
     Let $B(z)\neq z$ be a reducible Blaschke product with conjugate factor $\sigma=\sigma_a$. Then $B'(z)=0$ iff $z=-a$, and $-a$ is a fixed point of $B$.
\end{prop}
\begin{proof}
For the first claim, by the chain rule we have 
$$B'(z)=(\sigma'\circ h\circ\sigma^{-1})(z) \cdot (h'\circ\sigma^{-1})(z) \cdot (\sigma^{-1})'(z)$$
Because $\sigma$ is biholomorphic, it in particular does not admit any critical points and thus $(\sigma^{-1})'$ and $\sigma'\circ h\circ\sigma^{-1}$ have no zeroes. We have $h'(\sigma^{-1}(z))=n\sigma^{-1}(z)^{n-1}$, and this is zero iff $\sigma^{-1}(z)=0$, or equivalently if $z=-a$.

For the second claim, we have $\sigma_a^{-1}\circ B = h\circ\sigma_a^{-1}$. The zeroes of the left hand side are $B^{-1}(-a)$ while the zeroes of the right hand side are only $-a$, with multiplicity $n$. Thus $B^{-1}(-a)=\{-a\}$, showing that $-a$ is a fixed point.
\end{proof}

For a reducible Blaschke product $B$ of even degree, the conjugate factor can be related to the zeros of $B$ in terms of hyperbolic geometry. Recall that in the Poincar\'e disk, hyperbolic geodesics are circular arcs (resp. line segments) meeting the unit circle at right angles.
\begin{prop} \label{prop:reducible-blaschke-hyperbolic-geodesics}
    Let $B$ be a reducible Blaschke product of degree $2k$, $k\in\mathbb{N}$, with conjugate factor $\sigma_a$. Then $-a$ lies at the intersection of the hyperbolic geodesics through opposite pairs of zeroes, and the unsigned intersection angle of geodesics through consecutive pairs of roots of $B$ is identical to the unsigned intersection angle of geodesics through consecutive elements of $\Sigma_a^{2k}$.
\end{prop}
\begin{figure}[htp]
    \centering
    \begin{subfigure}[t]{0.32\linewidth}
         \centering
         \includegraphics[width=\textwidth]{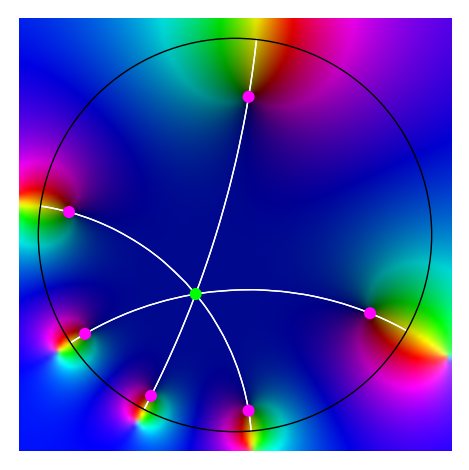}
         \caption{Hyperbolic geodesics (white) through opposite pairs of the zeros of $B$ (pink). The geodesics all intersect at $-a$ (green). }
     \end{subfigure}
    \hfill
    \begin{subfigure}[t]{0.32\linewidth}
         \centering
         \includegraphics[width=\textwidth]{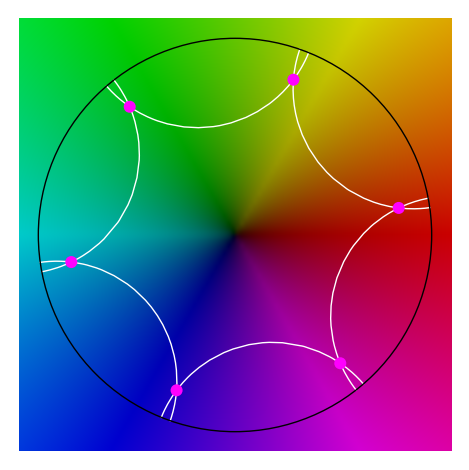}
         \caption{Hyperbolic geodesics (white) through consecutive pairs of elements of $\Sigma_a^{2k}$ (the $2k$-th roots of $a$) (pink).}
     \end{subfigure}
     \hfill
    \begin{subfigure}[t]{0.32\linewidth}
         \centering
         \includegraphics[width=\textwidth]{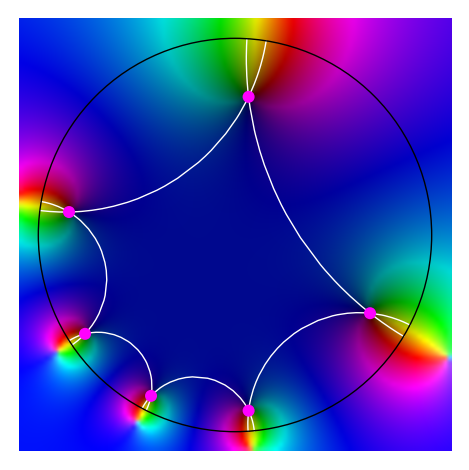}
         \caption{Hyperbolic geodesics (white) through consecutive pairs of the zeros of $B$ (pink). The angle at each zero matches the angle at each element of $\Sigma_a^{2k}$ in Subfigure (b).}
     \end{subfigure}
     \caption{Example of Proposition \ref{prop:reducible-blaschke-hyperbolic-geodesics} for a degree-$6$ reducible Blaschke product $B$ with $a = 0.2 + 0.3i$.}
    \label{fig:reducible-blaschke-hyperbolic-geodesics}
\end{figure}
\begin{proof}
Order the elements of $\Sigma_a^{2k}$ in the trigonometric orientation, that is, based on the principal argument 
\[0\leq\text{Arg}(\zeta^1f(a))<\text{Arg}(\zeta^{2}f(a))<\cdots<\text{Arg}(\zeta^{n}f(a))<2\pi.\] 
Let $\alpha,\beta$ be opposite pairs of zeroes, i.e. such that there exists $\omega\in\Sigma_a^{2k}$ with $\alpha=\sigma_a(\omega), \beta=\sigma_a(-\omega)$. The straight line though $\omega$ and $-\omega$ passes through $0$, and hence its image under $\sigma_a$ passes through $\sigma_a(0)=-a$. The image of the line must be a circle orthogonal to $\mathbb{T}$ passing through $\alpha,\beta$, showing that $\alpha,\beta$ and $-a$ are hyperbolically colinear. Letting $\omega$ range through the first $k$ elements of $\Sigma_{a}^{2k}$ we see that $-a$ lies at the intersection of the hyperbolic geodesics through opposite pairs of zeroes, proving the first claim. The second claim follows directly from conformality.
\end{proof}

We again point out that while Proposition \ref{prop:reducible-blaschke-hyperbolic-geodesics} deals with zeros of $B$, or in other words the preimages of $0$, the proof would work for intersections of hyperbolic geodesics connecting opposite pairs in any preimage set.

Recapitulating, we have found necessary conditions for a Blaschke product to be reducible in terms of its zeroes and critical points. However, given a reducible Blaschke product $B$, the product $\lambda B$ will satisfy these same conditions for all $\lambda\in\mathbb{T}$, but $\lambda B$ is only reducible if $\lambda = 1$. Thus it suffices to find the unimodular constant of our reducible Blaschke product.

\begin{thm}[Necessary and sufficient conditions for reducibility]\label{thm suff-neces-cond}
Let $B$ be a Blaschke product of degree $n$, and let $\mathcal{Z}$ be the set of (distinct) zeroes of $B$. Then $B$ is a reducible Blaschke product if and only if all of the following four conditions hold:
\begin{enumerate}
        \item $B$ has either one zero with multiplicity $n$, or $n$ zeroes with multiplicity one.
        \item $B$ has a unique critical point $-\xi$.
        \item $\sigma_{\xi}^{-1}(\mathcal{Z})=\Sigma_{\xi}^n$.
        \item $B(1)=\Delta_{\xi}$, where $\Delta_{\xi}=\frac{(1+\xi)^n-\xi(1+\overline{\xi})^n}{(1+\overline{\xi})^n-\xi(1+\xi)^n}$.
\end{enumerate}
In this case,
$$B(z)=\sigma_{\xi}\circ h\circ \sigma_{\xi}^{-1}=\Delta_{\xi}\prod_{w\in \mathcal{Z}}\left(\frac{(1-\overline{w})^2}{|1-w|^2}\sigma_w(z)\right).$$
    \end{thm}

\section{Concluding Remarks}\label{section with concluding remarks}

Another possible approach to studying Blaschke products is via algebraic invariants. Toward this end of connecting algebraic methods to geometric insights about Blaschke products, we initially considered the set
\[\mathcal{M}_{2\times 2}^B(\mathbb{T}) := \left\{M_a^\theta\in PGL(2,\mathbb{C}) : \exists\lambda_1,\lambda_2\in\mathbb{T}\quad \sigma_a^\theta(B^{-1}(\lambda_1))=B^{-1}(\lambda_2)\right\}\]
where $\sigma_a^\theta(z)=e^{i\theta}\frac{z-a}{1-\overline{a}{z}}$ and $M^\theta_a$ is the element of $PGL(2,\mathbb{C})$ associated to this Möbius transformation in the usual way. The idea behind this definition is to understand how automorphisms rearrange Blaschke $n$-gons in the unit circle. 
However, we found that because the set $\mathcal{M}_{2\times 2}^B(\mathbb{T})$ only contains automorphisms that shift one Blaschke $n$-gon to another, as illustrated in Example \ref{ex:pgl-non-group}, it does not have a group structure.

\begin{eg} \label{ex:pgl-non-group}
We use the result from Corollary 11.4 in \cite{Gorkin}. The result states that if $z_1,...,z_n$ and $w_1,...,w_n$ are two sets of points on $\mathbb{T}$ for which the principal arguments satisfy $0\leq \operatorname{arg}(z_1) < \operatorname{arg}(w_1)<\operatorname{arg}(z_2)<\operatorname{arg}(w_2)<...<\operatorname{arg}(z_n)<\operatorname{arg}(w_n)<2\pi$, then there exists a Blaschke product $B$ of degree $n$ such that $B(0)=0$, $B$ maps all the $z_j$ to the same image, and $B$ maps all the $w_j$ to the same image.

Take $n = 3$ for simplicity. Then the idea is to fix $z_1$, $z_2$, $z_3$, and $w_1$ and to change $w_2$ and $w_3$. To actually construct a counterexample, the same chapter of the book \cite[Chapter 11]{Gorkin} passes to the upper half plane model of the disk via the Möbius transformation $\phi(z):=i\left(\frac{1+z}{1-z}\right).$

We will construct a Blaschke product $B$ that maps the points $\zeta_{8}$,$\zeta_{8}\zeta_{3}$, and $\zeta_{8}\zeta_{3}^2$ to the same image and maps the points $i$, $-1$, and $\zeta_8^7$ to the same image, where $\zeta_8$ denotes the standard primitive $8$-th root of unity. We can then compare this to the Blaschke product $z^3$, which maps the first triple of points to a single image but not the second triple. That is, the Blaschke product $z^3$ has a Poncelet triangle with the elements of the first set as the vertices, but not one corresponding to the second set. We have $\phi(\zeta_{8})=-\frac{2\sqrt{2}}{4-\sqrt{2}}$, $\phi(\zeta_{8}\zeta_{3})=-\frac{2\sqrt{2-\sqrt{3}}}{4-\sqrt{2+\sqrt{3}}}$, $\phi(\zeta_{8}\zeta_3^2)=\frac{2\sqrt{2+\sqrt{3}}}{4-\sqrt{2-\sqrt{3}}}$, $\phi(i)=-1$, $\phi(-1)=0$, $\phi(\zeta_8^7)=\frac{2\sqrt{2}}{4-\sqrt{2}}$, and ordering them on the real line we have 
$$\phi(\zeta_8)<\phi(i)<\phi(\zeta_{8}\zeta_{3})<\phi(-1)<\phi(\zeta_8\zeta_3^2)<\phi(\zeta_8^7).$$
We can then construct the function $F$ from Lemma 11.3 of the same text \cite{Gorkin} to get 
$$F(z)=\frac{(z+1)z\left(z-\frac{2\sqrt{2}}{4-\sqrt{2}}\right)}{\left(z+\frac{2\sqrt{2}}{4-\sqrt{2}}\right)\left(z+\frac{2\sqrt{2-\sqrt{3}}}{4-\sqrt{2+\sqrt{3}}}\right)\left(z-\frac{2\sqrt{2+\sqrt{3}}}{4-\sqrt{2-\sqrt{3}}}\right)}$$
The Blaschke product we then want is $B=\phi^{-1}\circ F\circ \phi$.
\end{eg}

Because the set $\mathcal{M}_{2\times 2}^B(\mathbb{T})$ failed to carry a useful algebraic structure, we examined a more robust subset of disk automorphisms called the curve group of $B$. This group comprises the automorphisms that preserve all Blaschke $n$-gons, rather than just permuting a specific pair of them. The curve group is closely related to the group mentioned in \cite[Chapter 9]{Garcia} of continuous functions $u: \mathbb{T} \to \mathbb{T}$ satisfying $B \circ u = B$ on $\mathbb{T}$. However, we discovered that our findings were just corollaries of the more general results in \cite[Chapter 14]{Gorkin} and \cite{Aksoy}.

\section{Acknowledgements}
The authors are grateful to the anonymous referee for providing valuable feedback and suggestions which helped in the exposition of the paper and significantly shortened the proof of Theorem \ref{thm triangles}.
This project was part of the Polymath Jr. program in Summer 2022. We thank the program organizers for creating a platform to offer research experience opportunities to many students around the world.

\printbibliography

\end{document}